\documentclass[12pt]{article}

\usepackage{amssymb, amsmath, amsthm}
\usepackage{float}
\usepackage{lineno}

\usepackage{amsfonts}
\usepackage{amscd}
\usepackage{mathbbol}
\usepackage{hyperref}
\usepackage{comment}
\usepackage{footnote}
\numberwithin{equation}{section}

\theoremstyle{plain}

\newtheorem{thm}{Theorem}[section]
\newtheorem{proposition}[thm]{Proposition}
\newtheorem{lem}[thm]{Lemma}

\theoremstyle{definition}

\theoremstyle{remark}

\newtheorem*{rmk}{Remark}

\newcommand{\defi}{\operatorname{def}}
\newcommand{\D}{\operatorname{D}}
\newcommand{\M}{\operatorname{M}}

\usepackage{tikz}
\usetikzlibrary{arrows, decorations.markings}
\usepackage{authblk}
\begin{document}

\title{The Chromatic Number of Joins of Signed Graphs \footnote{This paper originates from a doctoral thesis written under the supervision of Thomas Zaslavsky.}}
\author{Amelia R.W. Mattern}
\affil{Binghamton University, Binghamton, New York, U.S.A.}
\maketitle

\begin{abstract}
We introduce joins of signed graphs and explore the chromatic number of the all-positive and all-negative joins. We prove an analogue to the theorem that the chromatic number of the join of two graphs equals the sum of their chromatic numbers. Given two signed graphs, the chromatic number of the all-positive and all-negative join is usually less than the sum of their chromatic numbers, by an amount that depends on the new concept of deficiency of a signed-graph coloration.
\end{abstract}

\section{Introduction}
A \emph{signed graph} is a graph in which every edge has an associated sign. We write a signed graph $\Sigma$ as the triple $(V, E, \sigma)$ where $V$ is the vertex set, $E$ is the edge set, and $\sigma: E \to \{+,-\}$ is the \emph{signature}. Our graphs are \emph{signed simple graphs}: with no loops and no multiple edges. 

We define signed-graph coloring as in \cite{zaslavsky}, and chromatic number as in \cite{macajova}. A \emph{proper coloration} of a signed graph $\Sigma$ is a function, $\kappa : V \to \{\pm1, \pm 2, \ldots, \allowbreak \pm k, 0\},$ such that for any edge $e_{ab} \in E$, $\kappa(a) \neq \sigma(e) \kappa(b).$ The \emph{chromatic number} of $\Sigma$, written $\chi(\Sigma)$, is the size of the smallest set of colors which can be used to properly color $\Sigma.$ A graph with chromatic number $k$ is called \emph{k-chromatic}. A coloration is \emph{minimal} if it is proper and uses a set of colors of size $\chi(\Sigma).$ If $\chi = 2k$ then a minimal \emph{color set} is $\{\pm1, \pm2, \ldots, \pm k\},$ and if $\chi = 2k+1$ then a minimal color set is $\{\pm1, \pm2, \ldots, \pm k, 0\}.$ If $\chi = 2k+1$, there must be at least one vertex colored 0 in every minimal coloration.

The \emph{deficiency} of a coloration, $\defi(\kappa)$, is the number of unused colors from the color set of $\kappa.$ The \emph{deficiency set}, $\D(\kappa)$, is the set of unused colors. The \emph{maximum deficiency} of a graph, $\M(\Sigma) $, is $\max\{\defi(\kappa) \mid \kappa$ is a minimal proper coloration of $\Sigma \}.$ For more about deficiency, see \cite{mattern}.

Let $\Sigma_1$ and $\Sigma_2$ be signed graphs. The \emph{$\sigma^*$-join} of $\Sigma_1$ and $\Sigma_2$, written $\Sigma_1 \vee_{\sigma^*} \Sigma_2$, is the signed graph with
\begin{align*}
V &= V(\Sigma_1) \cup V(\Sigma_2), \\
E &= E(\Sigma_1) \cup E(\Sigma_2) \cup \{e_{vw} \mid v \in V(\Sigma_1), w \in V(\Sigma_2)\}, \\
\text{and } \sigma(e) &= \begin{cases}
\sigma_1(e) &\text{ if } e \in E(\Sigma_1), \\
\sigma_2(e) &\text{ if } e \in E(\Sigma_2), \\
\sigma^*(e) &\text{ otherwise,}
\end{cases}
\end{align*}
where $\sigma^*$ is a function from the new join edges to the set $\{+,-\}.$ 

We explore joins of signed graphs and prove an analogue to the theorem that the chromatic number of the join of two graphs equals the sum of their chromatic numbers. In this case, the chromatic number of the join of two signed graphs depends on both the chromatic numbers and the maximum deficiencies of the two graphs. 

\section{Joins of Signed Graphs}
\label{joins of signed graphs}
Unlike in ordinary graph theory, there are many ways to join two signed graphs. The result of joining two signed graphs depends on the signs of the new edges. In this paper we focus on the \emph{all-positive join} of two signed graphs, where each new join edge has positive sign. We write the all-positive join of $\Sigma_1$ and $\Sigma_2$ as $\Sigma_1 \vee_+ \Sigma_2.$ 

Let $\Sigma$ be a signed graph with even chromatic number and maximum deficiency $M$. Then $\Sigma$ is an {\it exceptional graph} if in every proper coloration using $\chi(\Sigma) - M$ colors, every color that is used appears at both ends of some negative edge. Figure 1 shows two examples of exceptional graphs. Graph A is a 6-chromatic graph with maximum deficiency 3. Graph B is a 4-chromatic graph with maximum deficiency 2. In both colorations, every color used appears on both endpoints of some edge. In depictions we use solid lines for positive edges and dashed lines for negative edges. 

\begin{figure}[H]
\centering
\label{exceptional}
\begin{tikzpicture}[
thick,
       acteur/.style={
         circle,
         fill=black,
         thick,
         inner sep=2pt,
         minimum size=0.2cm
       }
    ] 
\node at (0,0) [acteur,label = below:1]{};
\node at (3,0) [acteur,label = below:-2]{};
\node at (0,4) [acteur,label = above:3]{};
\node at (3,4) [acteur,label = above:3]{};
\node at (-1,2) [acteur,label = left:1]{};
\node at (4,2) [acteur,label = right: -2]{};

\node at (8,0) [acteur, label = below: 1]{};
\node at (10,0) [acteur, label = below: 2]{};
\node at (12,0) [acteur, label = below: 2]{};
\node at (8,2) [acteur, label = above: 1]{};
\node at (10,2) [acteur, label = above: 1]{};
\node at (12,2) [acteur, label = above: 1]{};

\draw[thick] (0,0) -- (3,0) -- (0,4) -- (-1,2) -- (4,2) -- (3,4) -- cycle;
\draw[thick, dashed] (0,4) -- (3,4) -- (-1,2) -- (0,0) -- cycle;
\draw[thick, dashed] (3,0) -- (4,2) -- (0,0);
\draw[thick, dashed] (3,0) -- (3,4);
\draw[thick, dashed] (3,0) -- (-1,2);
\draw[thick, dashed] (0,4) -- (4,2);

\draw[thick] (10,0) -- (10,2);
\draw[thick] (12,0) -- (12,2);
\draw[thick, dashed] (10,0) -- (8,0) -- (8,2) -- (10,2) -- (8,0);
\draw[thick, dashed] (10,0) -- (8,2);
\draw[thick, dashed] (10,0) -- (12,0) -- (12,2) -- (10,2) -- (12,0);
\draw[thick, dashed] (10,0) -- (12,2);
\node at (1.5, -2) [label = Graph A]{};
\node at (10,-2) [label = Graph B]{};
\end{tikzpicture}
\caption[Two properly colored exceptional graphs]{Two properly colored exceptional graphs.}
\end{figure}
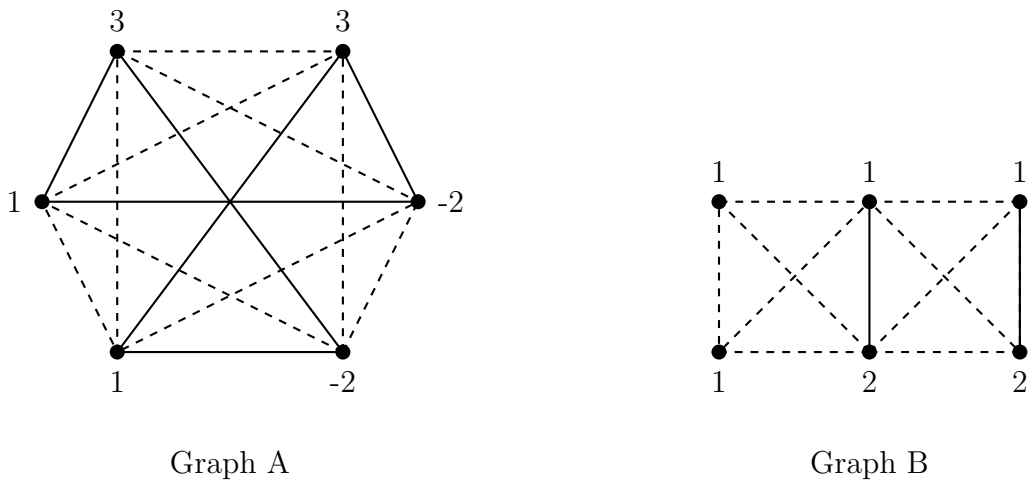

There exist infinitely many exceptional graphs. For example, for non-negative $k,$ if $\chi = 2k,$ the complete graph on $4k$ vertices with a negative perfect matching and all other edges positive is an exceptional graph with maximum deficiency 0. It remains an open problem to characterize exceptional graphs in terms of their structure.

\begin{thm}
Let $\Sigma_1$ and $\Sigma_2$ be signed graphs with maximum deficiencies $M_1$ and $M_2,$ respectively. Assume that $M_1 \geq M_2.$ Then, with one exception, 
$$\chi(\Sigma_1 \vee_+ \Sigma_2) = \max\{\chi_1 + \chi_2 - M_1 - M_2, \chi_1\}.$$

Exception: If  $\Sigma_1$ and $\Sigma_2$ both have even chromatic number, exactly one of $M_1$ and $M_2$ is odd, and both $\Sigma_1$ and $\Sigma_2$ are exceptional graphs, then 
$$\chi(\Sigma_1 \vee_+ \Sigma_2) = \max\{\chi_1 + \chi_2 - M_1 - M_2 +1, \chi_1\}.$$
\label{main thm}
\end{thm}

Note that in Theorem \ref{main thm},  $\chi_2$ can never be larger than $\chi_1 + \chi_2 - M_1 - M_2.$ If $\chi_2 > \chi_1 + \chi_2 - M_1 - M_2,$ this would imply that $M_2 > \chi_1 - M_1.$ Since $M_1 \leq \frac{1}{2}\chi_1,$ we then would have $M_2 > \frac{1}{2} \chi_1 \geq M_1$. This contradicts our assumption that $M_1 \geq M_2.$

Let $A$ be a set of vertices of a signed graph $\Sigma.$ \emph{Switching} $A$ is negating the signs of the edges with exactly one endpoint in $A$. Two signed graphs are \emph{switching equivalent} if they are related by switching. The chromatic number of a signed graph is the same as the chromatic number of the switched graph. A minimal coloration of the switched graph is simply a byproduct of switching the graph; given a minimal coloration of $\Sigma$, the signs of the colors on $A$ are negated during the switching process. Two colorations of $\Sigma, \kappa$ and $\kappa^*,$ are \emph{switching equivalent} if they are related by switching.

\begin{rmk}
Theorem \ref{main thm} also holds for the all-negative join of signed graphs. The possible joins of $\Sigma_1$ and $\Sigma_2$ come in switching equivalent pairs. For a signature $\sigma,$ let $-\sigma$ be the signature in which the sign of an edge $e$ is $-\sigma(e).$ 
Then $\Sigma_1 \vee_{\sigma^*} \Sigma_2$ switches to $\Sigma_1 \vee_{-\sigma^*} \Sigma_2$. Thus, $\Sigma_1 \vee_- \Sigma_2$ switches to $\Sigma_1 \vee_+ \Sigma_2$ by switching the vertices of $\Sigma_1.$ Furthermore, switching all the vertices of $\Sigma_1$ does not change the signs of the edges in $\Sigma_1$ or $\Sigma_2$, and thus does not change the maximum deficiencies.
\end{rmk}

\section{Preliminaries}
\label{preliminaries}

When coloring the graph $\Sigma_1 \vee_+ \Sigma_2,$ one cannot use the same color on both the vertices of $\Sigma_1$ and $\Sigma_2$. This gives a straightforward lower bound for the chromatic number of the all-positive join.

\begin{lem}
Let $\Sigma_1$ and $\Sigma_2$ be signed graphs with maximum deficiencies $M_1$ and $M_2$, respectively. Then $\chi(\Sigma_1 \vee_+ \Sigma_2) \geq \chi_1 + \chi_2 - M_1 - M_2.$
\label{lower bound}
\end{lem}

Throughout the proofs of Theorem \ref{main thm}, we use several recoloration tools in order to give a proper coloration of the correct size. Define the following replacement types:

{\bf Type 1:} Let $r \in \mathbb Z \setminus \{0\}.$ For $i$ in the color set of $\kappa,$ recolor the $i$-color set with the color $r.$

{\bf Type 2:} For $i \neq 0$ in the color set of $\kappa,$ such that $i$ does not appear on both endpoints of an edge, recolor the $i$-color set with the color 0.

{\bf Type 3:} Let $r \in \mathbb Z \setminus \{0\}.$ For $i$ and $-i$ in the color set of $\kappa,$ such that $i \neq 0$, recolor the $i$-color set with the color $r$ and the $(-i)$-color set with the color $-r$.

{\bf Type 4:} Let $r_1, r_2 \in \mathbb Z \setminus \{0\}$ such that $r_1 \neq -r_2.$ For $i$ and $-i$ in the color set of $\kappa,$ such that $i \neq 0$, recolor the $i$-color set with the color $r_1$ and the $(-i)$-color set with the color $r_2$.

\begin{proposition}
Let $\Sigma$ be a signed graph with even chromatic number and non-zero maximum deficiency. Then in every minimal coloration with maximum deficiency, the negative of every color in the deficiency set appears on both endpoints of some negative edge.
\label{even chromatic}
\end{proposition}

\begin{proof}
Let $\Sigma$ be a signed graph with even chromatic number $\chi$ and non-zero maximum deficiency $M$. Let $\kappa$ be a coloration of $\Sigma$ such that $\defi(\kappa) = M.$ Suppose to the contrary that there exists some $i \in \D(\kappa)$ such that $-i$ does not appear on both endpoints of a negative edge. Thus, if $\kappa(v) = -i$, then no neighbor of $v$ is colored $-i.$ Recoloring all vertices colored $-i$ with the color 0 yields a proper coloration since $\kappa$ did not use the color 0 and no two vertices colored 0 are adjacent. But, the size of the new color set is $\chi - 1.$
\end{proof}

\section{Chromatic Number $\chi_1$ and the Exceptional Case}
\label{chi_1 and exception}

\begin{lem}
If $\chi_1 \geq \chi_1 + \chi_2 - M_1 -M_2,$ then $\chi(\Sigma_1 \vee_+ \Sigma_2) = \chi_1.$
\label{exception 1}
\end{lem}

\begin{proof}
Let $\kappa$ be a coloration of $\Sigma_1 \vee_+ \Sigma_2.$ Then $\kappa$ restricted to $\Sigma_1$ is a proper coloration of $\Sigma_1.$ Therefore, the size of the color set must be at least $\chi_1.$

Now we show there exists a proper coloration of $\Sigma_1 \vee_+ \Sigma_2$ using a color set of size $\chi_1$. Color $\Sigma_1 \vee_+ \Sigma_2$ in the following way.
\begin{enumerate}
\item Properly color $\Sigma_1$ using $\chi_1 - M_1$ colors. Call this coloration $\kappa$.
\item Properly color $\Sigma_2$ using the colors in the deficiency set of $\kappa$. This is possible since the deficiency set is made up of $M_1$ colors with distinct absolute values, and $M_1 > \chi_2 - M_2$, because $\chi_1 > \chi_1 + \chi_2 - M_1 - M_2.$ 
\end{enumerate}  

This coloration is proper on $\Sigma_1$ and $\Sigma_2$ by definition. Since every join edge is positive, and every color used on $\Sigma_2$ does not appear on a vertex of $\Sigma_1$, the coloration is proper on $\Sigma_1 \vee_+ \Sigma_2.$ Furthermore, the size of the color set is $\chi_1,$ and therefore, $\chi(\Sigma_1 \vee_+ \Sigma_2) = \chi_1.$
\end{proof}

\begin{lem}[Exception]
Assume that $\chi_1$ and $\chi_2$ are even, $M_1 > M_2,$ and $\chi_1 \leq \chi_1 + \chi_2 - M_1 - M_2 + 1.$ If exactly one of $M_1$ and $M_2$ is odd, and both $\Sigma_1$ and $\Sigma_2$ are exceptional, then $\chi(\Sigma_1 \vee_+ \Sigma_2) = \chi_1 + \chi_2 - M_1 - M_2 +1.$
\label{exception 2}
\end{lem}

\begin{proof}
Let $\chi_1 = 2k_1$ and $\chi_2 = 2k_2$ for some positive integers $k_1$ and $k_2.$ By Lemma \ref{lower bound} we know that $\chi(\Sigma_1 \vee_+ \Sigma_2) \geq \chi_1 + \chi_2 - M_1 - M_2.$ 

Suppose that $\chi(\Sigma_1 \vee_+ \Sigma_2) = \chi_1 + \chi_2 - M_1 - M_2.$ Note that $\chi_1 + \chi_2 - M_1 - M_2$ is odd. Thus a minimal coloration of $\Sigma_1 \vee_+ \Sigma_2$ must use 0 as a color. Let $\kappa$ be such a coloration. Because all join edges are positive, no color can be used on both $\Sigma_1$ and the vertices of $\Sigma_2.$ Thus, $\kappa$ must use $2k_1 - M_1$ colors on $\Sigma_1$ and $2k_2 - M_2$ different colors on the vertices of $\Sigma_2.$ Therefore $\kappa$ restricted to $\Sigma_i$ is a proper coloration using $\chi_i - m_i$ colors. Since both $\Sigma_1$ and $\Sigma_2$ are exceptional graphs, every color must appear on both endpoints of some negative edge. But the color 0 cannot appear on both endpoints of an edge in a proper coloration. Therefore, $\chi(\Sigma_1 \vee_+ \Sigma_2) > \chi_1 + \chi_2 - M_1 - M_2.$

Now we show that there exists a proper coloration of $\Sigma_1 \vee_+ \Sigma_2$ using a color set of size $\chi_1 + \chi_2 - M_1 - M_2 + 1$. Although the size of the color set will be $\chi_1 + \chi_2 - M_1 - M_2 + 1$, we will only use $\chi_1 + \chi_2 - M_1 - M_2$ colors. Color $\Sigma_1 \vee_+ \Sigma_2$ in the following way.
\begin{enumerate}
\item Properly color $\Sigma_1$ with colors $\pm 1, \pm 2, \ldots, \pm k_1$ using $2k_1 - m_1$ colors. Let the $M_1$ unused colors be $x_1, x_2, \ldots, x_{M_1}$. 
\item Properly color $\Sigma_2$ with colors $\pm 1, \pm 2,  \ldots, \pm k_2$ using $2k_2 - M_2$ colors. Let the $M_2$ unused colors be $y_1, y_2, \ldots, y_{M_2}$.
\item Make the following Type 1 color replacements on the vertices of $\Sigma_2:$
\begin{center}
\begin{tabular}{| c | c |}
\hline
Old Color & New Color \\ \hline
$-y_1$ & $x_1$ \\ \hline
$-y_2$ & $x_2$ \\ \hline
$\vdots$ & $\vdots$ \\ \hline
$-y_{M_2}$ & $x_{M_2}$ \\
\hline
\end{tabular}
\end{center}

\item Using Type 4 replacements, recolor $\frac{M_1 - M_2 - 1}{2}$ more pairs of color in $\Sigma_2$ using colors $x_{M_2 + 1}, \ldots, x_{M_1 - 1}.$ Note that we only go up to $x_{M_1 - 1}$ since we need an even number of colors for this step.
\item Replace the remaining pairs of colors in $\Sigma_2$ using Type 3 replacements and colors $\pm(k_1 + 1), \ldots, \pm (k_1 + k_2 - \frac{M_1 - M_2 - 1}{2}).$
\end{enumerate}

Call the new coloration $\kappa.$ Since no colors were changed in $\Sigma_1$, $\kappa$ is proper on $\Sigma_1.$ In $\Sigma_2$ all of the color sets were recolored using Type 1, Type 3, and Type 4 replacements, and the original partition of $V(\Sigma_2)$ into color sets was maintained. Thus, the subgraph induced by every color set of $\kappa$ on $\Sigma_2$ is all-negative. Furthermore, if $r$ is a color of $\kappa$ resulting from a Type 1 or Type 4 replacement, then there are no vertices colored $-r$ in $\Sigma_2.$ If $r$ and $-r$ are colors resulting from a Type 3 replacement, then there are no negative edges between the two color sets.
Therefore, $\kappa$ is proper on $\Sigma_2.$ Finally, no color is used on the vertices of both $\Sigma_1$ and $\Sigma_2$. Thus, $\kappa$ is proper on $\Sigma_1 \vee_+ \Sigma_2.$ Furthermore, the color set is
$\{\pm1, \ldots, \pm (k_1 + k_2 - \frac{M_1-M_2-1}{2})\}.$
Therefore, $$\chi(\Sigma_1 \vee_+ \Sigma_2) = 2\left(k_1 + k_2 - \frac{M_1-M_2-1}{2}\right) = \chi_1 + \chi_2 - M_1 - M_2 + 1. \qedhere$$
\end{proof}

\begin{rmk}
In the case where $M_2 = 0,$ one would simply skip step 3 in the recoloration procedure.
\end{rmk}

Because it is similar to the proof in the exceptional case, for the remaining cases we leave the proof that the new coloration is both proper and of correct size as an exercise for the reader.

\section{Non-Exceptional Cases}
\label{Non-exceptions}

\begin{lem}[$M_2 = 0$]
Let $M_2 = 0.$ Assume $\Sigma_1$ and $\Sigma_2$ do not satisfy the conditions of the exception and that $\chi_1 < \chi_1 + \chi_2 - M_1 - M_2$. Then $\chi(\Sigma_1 \vee_+ \Sigma_2) = \chi_1 + \chi_2 - M_1 - M_2.$
\end{lem}
\label{main lemma 0}

\begin{proof}
We need only show a proper coloration using a color set of size $\chi_1 + \chi_2 - M_1 - M_2.$ We have two cases to consider.

{\bf Case 1:} Suppose $M_1$ is even. We have two subcases.

{\bf Case 1.1:} Suppose at least one of $\chi_1$ and $\chi_2$ is even. Then color $\Sigma_1 \vee_+ \Sigma_2$ in the following way. 
\begin{enumerate}
\item Properly color $\Sigma_1$ with $\pm1, \ldots, \pm k_1,$ and $0$ if $\chi_1$ is odd, using $\chi_1 -M_1$ colors. Let the $M_1$ unused colors be $x_1, \ldots, x_{M_1}.$ If $\chi_1$ is odd, 0 must be used.
\item Properly color $\Sigma_2$ with $\pm 1, \ldots, \pm k_2,$ and $0$ if $\chi_1$ is even and $\chi_2$ is odd, using $\chi_2$ colors.
\item Make the following Type 1 and Type 3 replacements on $\Sigma_2.$
\begin{center}
  \begin{tabular}{ | c | c || c | c | }
    \hline
    Old Color & New Color & Old Color & New Color\\ \hline
    1 & $x_1$ & $\frac{M_1}{2} + 1$ & $k_1 +1$ \\ \hline
    -1 & $x_2$ & $-(\frac{M_1}{2} + 1)$ & $-(k_1 +1)$ \\ \hline
    $\vdots$ & $\vdots$ & $\vdots$ & $\vdots$  \\ \hline
    $\frac{M_1}{2}$ & $x_{M_1 -1}$ & $k_2$ & $k_1 + k_2 - \frac{M_1}{2}$ \\ \hline 
    $-\frac{M_1}{2}$ & $x_{M_1}$ & $-k_2$ & $-(k_1 + k_2 - \frac{M_1}{2})$ \\ 
    \hline
  \end{tabular}
\end{center}
\end{enumerate}

{\bf Case 1.2:} Suppose both $\chi_1$ and $\chi_2$ are odd. That is, $\chi_1 = 2k_1 + 1$ and $\chi_2 = 2k_2 + 1$ for some positive integers $k_1$ and $k_2.$ Then create the coloration $\kappa$ in the following way:
\begin{enumerate}
\item Properly color $\Sigma_1$ with $\pm1, \ldots, \pm k_1, 0$ using $\chi_1 -M_1$ colors. Let the $M_1$ unused colors be $x_1, \ldots, x_{M_1}.$ Note that 0 must be used.
\item Properly color $\Sigma_2$ with $\pm 1, \ldots, \pm k_2, 0$ using $\chi_2$ colors.
\item Make the following Type 1 and Type 3 replacements on $\Sigma_2.$ 

\begin{center}
\begin{tabular}{ | c | c || c | c | }
	\hline
	Old Color & New Color & Old Color & New Color \\ \hline
	1 &  $x_1$  & $\frac{M_1}{2} + 1$ &  $k_1 + 1$ \\ \hline
	-1 & $x_2$  & $-(\frac{M_1}{2} + 1)$ & $-(k_1 + 1)$ \\ \hline
	$\vdots$ & $\vdots$ & $\vdots$ & $\vdots$ \\ \hline
	$\frac{M_1}{2}$ & $x_{M_1 -1}$ & $k_2$ & $k_1 + k_2 - \frac{M_1}{2}$ \\ \hline
	$-\frac{M_1}{2}$ & $x_{M_1}$ & $-k_2$ & $-(k_1 + k_2 - \frac{M_1}{2})$ \\
	\hline
\end{tabular}
\end{center}

\item Use Type 1 replacements to recolor the 0-color set in $\Sigma_1$ with $k_1 + k_2 - \frac{M_1}{2} + 1$ and to recolor the 0-color set in $\Sigma_2$ with $-(k_1 + k_2 - \frac{M_1}{2} + 1).$
\end{enumerate}

{\bf Case 2:} Suppose $M_1$ is odd. We have several cases.

{\bf Case 2.1:} Suppose both $\chi_1$ and $\chi_2$ are even. Let $\chi_1 = 2k_1$ and $\chi_2 = 2k_2$ for some positive integers $k_1$ and $k_2$. Then either $\Sigma_1$ or $\Sigma_2$ is not exceptional.

{\bf Case: 2.1a:} Suppose $\Sigma_1$ is not exceptional. Then there exists a coloration using $\chi_1 - M_1$ colors such that there is no edge with both endpoints colored $a,$ for some $a$ in the set of colors used. Call this coloration $\kappa_1.$ Then color $\Sigma_1 \vee_+ \Sigma_2$ in the following way:

\begin{enumerate}
\item Color $\Sigma_1$ with $\kappa_1$. Let the $M_1$ unused colors be $x_1, \ldots, x_{M_1}$. By choice of notation, let $a = 1.$
\item Properly color $\Sigma_2$ with $\pm 1, \ldots, \pm k_2$ using $\chi_2$ colors.
\item Use a Type 2 replacement to replace $a$ with 0 in $\Sigma_1.$
\item Make the following Type 1 replacements on $\Sigma_2.$
\begin{center}
\begin{tabular}{| c | c |}
\hline
Old Color & New Color  \\ \hline
$-1$ & $x_1$  \\ \hline
$2$ & $x_2$  \\ \hline
$\vdots$ & $\vdots$ \\ \hline
$\frac {M_1+1}{2}$ & $x_{M_1-1}$ \\ \hline
$-\frac{M_1+1}{2}$ & $x_{M_1}$ \\
\hline
\end{tabular}
\end{center}
\item Use Type 3 replacements to recolor the remaining $k_2 - \frac{M_1 + 1}{2}$ pairs of colors in $\Sigma_2$ with colors $\pm (k_1+1), \ldots, \pm (k_1 + k_2 -\frac{M_1 + 1}{2}) .$
\end{enumerate}

{\bf Case 2.1b:} 
Suppose $\Sigma_2$ is not exceptional. Then there exists a proper coloration such that no edge has both endpoints colored $a$ for some $a$ in the color set. Call this coloration $\kappa_2.$ Then color $\Sigma_1 \vee_+ \Sigma_2$ in the following way:

\begin{enumerate}
\item Properly color $\Sigma_1$ with $\pm1, \ldots, \pm k_1$ and using $\chi_1-M_1$ colors. Let the $M_1$ unused colors be $x_1, \ldots, x_{M_1}.$ 
\item Color $\Sigma_2$ with $\kappa_2$ using colors $\pm 1, \ldots, \pm k_2$ using $\chi_2$ colors. By choice of notation, let $a = 1.$
\item Make the following Type 1, Type 2, and Type 3 replacements on  $\Sigma_2.$
\begin{center}
\begin{tabular}{| c | c | | c | c |}
\hline
Old Color & New Color & Old Color & New Color \\ \hline
1 & 0 & $\frac{M_1+1}{2}$ & $k_1 + 1$ \\ \hline
-1 & $x_{M_1}$ & $-\frac{M_1+1}{2}$ & $-(k_1 + 1)$ \\ \hline
2 & $x_1$ & $\vdots$ & $\vdots$ \\ \hline
-2 & $x_2$ & $-k_2$  & $-(k_1 + k_2 - \frac{M_1+1}{2})$ \\ \hline
$\vdots$ & $\vdots$ & & \\ \hline
$-\left(\frac{M_1-1}{2}\right)$ & $x_{M_1 - 1}$ & & \\
\hline
\end{tabular}
\end{center}
\end{enumerate}

{\bf Case 2.2:} Suppose $\chi_1 = 2k_1 + 1$ and $\chi_2 = 2k_2$ for some positive integers $k_1$ and $k_2$. Then color $\Sigma_1 \vee_+ \Sigma_2$ in the following way:

\begin{enumerate}
\item Properly color $\Sigma_1$ with $\pm1, \ldots, \pm k_1, 0$ using $\chi_1-M_1$ colors. Let the $M_1$ unused colors be $x_1, \ldots, x_{M_1}.$ Note that $0$ must be used.
\item Properly color $\Sigma_2$ with $\pm 1, \ldots, \pm k_2$ using $\chi_2$ colors.
\item In $\Sigma_1$ use a Type 1 replacement to replace 0 with $-(k_1+1).$
\item Make the following Type 1 and Type 3 replacements on $\Sigma_2.$

\begin{center}
\begin{tabular}{| c | c || c | c |}
\hline
Old Color & New Color & Old Color & New Color \\ \hline
1 & $x_1$ & $-\frac{M_1+1}{2}$ & $k_1 + 1$ \\ \hline
-1 & $x_2$ & $\frac{M_1 + 2}{2}$ & $k_1 + 2$  \\ \hline
$\vdots$ & $\vdots$ & $-\frac{M_1 + 2}{2}$ & $-(k_1 + 2)$ \\ \hline
$-\left(\frac{M_1-1}{2}\right)$ & $x_{M_1 - 1}$ & $\vdots$ & $\vdots$   \\ \hline
$\frac{M_1+1}{2}$ & $x_{M_1}$ & $-k_2$ & $-(k_1 + k_2 - \frac{M_1-1}{2})$ \\
\hline
\end{tabular}
\end{center}
\end{enumerate}

{\bf Case 2.3:} Suppose $\chi_2 = 2k_2 + 1$ for some positive integer $k_2$. Then color $\Sigma_1 \vee_+ \Sigma_2$ in the following way:

\begin{enumerate}
\item Properly color $\Sigma_1$ with $\pm1, \ldots, \pm k_1,$ and 0 if $\chi_1$ is also odd, using $\chi_1-M_1$ colors. Let the $M_1$ unused colors be $x_1, \ldots, x_{M_1}.$ If $\chi_1$ is odd, 0 must be used.
\item Properly color $\Sigma_2$ with $\pm 1, \ldots, \pm k_2, 0$ using $\chi_2$ colors.
\item Make the following Type 1 and Type 3 replacements on $\Sigma_2.$
\begin{center}
\begin{tabular}{| c | c || c | c |}
\hline
Old Color & New Color & Old Color & New Color \\ \hline
1 & $x_1$ & $\frac{M_1+1}{2}$ & $k_1 + 1$ \\ \hline
-1 & $x_2$ & $-\frac{M_1+1}{2}$ & $-(k_1 + 1)$ \\ \hline
$\vdots$ & $\vdots$ & $\vdots$ & $\vdots$ \\ \hline
$-\left(\frac{M_1-1}{2}\right)$ & $x_{M_1 - 1}$ & $-k_2$ & $-(k_1 + k_2 - \frac{M_1-1}{2})$ \\ \hline
$0$ & $x_{M_1}$  & & \\
\hline
\end{tabular}
\end{center}
\end{enumerate}

This concludes the non-exceptional cases where $M_2 = 0.$.
\end{proof}

\begin{lem} [$M_2 > 0$]
Let $M_2 > 0.$ Assume that $M_1 \geq M_2,$ $\Sigma_1$ and $\Sigma_2$ are not exceptional, and $\chi_1 < \chi_1 + \chi_2 - M_1 - M_2$. Then $\chi(\Sigma_1 \vee_+ \Sigma_2) = \chi_1 + \chi_2 - M_1 - M_2.$
\end{lem}
\label{main lemma non-0}

\begin{proof}
We need only show a coloration using a color set of size $\chi_1 + \chi_2 - M_1 - M_2.$ We have two cases to consider.

{\bf Case 1:} Suppose $M_1$ and $M_2$ are either both even or both odd. This implies $M_1 - M_2$ is even.

{\bf Case 1.1:} Suppose $\chi_1 = 2k_1 + 1$ and $\chi_2 = 2k_2 + 1$ for some positive integers $k_1$ and $k_2.$ Then color $\Sigma_1 \vee_+ \Sigma_2$ in the following way:

\begin{enumerate}
\item Properly color $\Sigma_1$ with $\pm1, \ldots, \pm k_1, 0$ using $\chi_1-M_1$ colors. Let the $M_1$ unused colors be $x_1, \ldots, x_{M_1}.$ Note that 0 must be used.
\item Properly color $\Sigma_2$ with $\pm 1, \ldots, \pm k_2, 0$ using $\chi_2 - M_2$ colors. Let the $M_2$ unused colors be $y_1, \ldots, y_{M_2}.$ Note that $0$ must be used.
\item In $\Sigma_1$ perform a Type 1 replacement and replace 0 with $k_1 + 1.$
\item Make the following Type 1 replacements on $\Sigma_2.$
\begin{center}
\begin{tabular}{| c | c |}
\hline
Old Color & New Color \\ \hline
$0$ & $-(k_1 +1)$ \\ \hline
$-y_1$ & $x_1$ \\ \hline
$-y_2$ & $x_2$ \\ \hline
$\vdots$ & $\vdots$ \\ \hline
$-y_{M_2}$ & $x_{M_2}$ \\
\hline
\end{tabular}
\end{center}
\item Replace $\frac{M_1 - M_2}{2}$ more pairs of colors in $\Sigma_2$ using Type 4 replacements and colors $x_{M_2 + 1}, \ldots, x_{M_1}$.
\item Use Type 3 replacements to recolor the remaining pairs of colors in $\Sigma_2$ using $\pm (k_1 + 2), \ldots, \pm(k_1 + k_2 - M_2 - \frac{M_1 - M_2}{2} + 1).$
\end{enumerate}

{\bf Case 1.2:} Suppose at least one of $\chi_1$ and $\chi_2$ is even. Then color $\Sigma_1 \vee_+ \Sigma_2$ in the following way:

\begin{enumerate}
\item Properly color $\Sigma_1$ with $\pm1, \ldots, \pm k_1,$ and possibly 0 using $\chi_1-M_1$ colors. Let the $M_1$ unused colors be $x_1, \ldots, x_{M_1}.$ If $\chi_1$ is odd, 0 must be used.
\item Properly color $\Sigma_2$ with $\pm 1, \ldots, \pm k_2,$ and possibly 0 using $\chi_2 - M_2$ colors. Let the $M_2$ unused colors be $y_1, \ldots, y_{M_2}.$ If $\chi_2$ is odd, $0$ must be used.
\item Make the following Type 1 replacements on $\Sigma_2.$
\begin{center}
\begin{tabular}{| c | c |}
\hline
Old Color & New Color \\ \hline
$-y_1$ & $x_1$ \\ \hline
$-y_2$ & $x_2$ \\ \hline
$\vdots$ & $\vdots$ \\ \hline
$-y_{M_2}$ & $x_{M_2}$ \\
\hline
\end{tabular}
\end{center}
\item Use Type 4 replacements to recolor $\frac{M_1 - M_2}{2}$ more pairs of colors in $\Sigma_2$ using $x_{M_2 + 1}, \ldots, x_{M_1}$.
\item Replace the remaining pairs of colors in $\Sigma_2$ using Type 3 replacements and colors $\pm (k_1 + 1), \ldots, \pm(k_1 + k_2 - M_2 - \frac{M_1 - M_2}{2}).$
\end{enumerate}

{\bf Case 2:} Now suppose exactly one of $M_1$ and $M_2$ is odd. This implies that $M_1 - M_2$ is odd.

{\bf Case 2.1:} Suppose $\chi_1 = 2k_1$ and $\chi_2 = 2k_2$ for some positive integers $k_1$ and $k_2.$ Then either $\Sigma_1$ or $\Sigma_2$ is not exceptional. 

{\bf Case 2.1a:} Suppose $\Sigma_1$ is not exceptional. Then there exists a coloration using $\chi_1 - M_1$ colors such that there is no edge with both endpoints colored $a,$ for some $a$ in the set of colors used. Call this coloration $\kappa_1.$ Then color $\Sigma_1 \vee_+ \Sigma_2$ in the following way:

\begin{enumerate}
\item Color $\Sigma_1$ using $\kappa_1.$ Let the $M_1$ unused colors be $x_1, \ldots, x_{M_1}.$ By choice of notation, let $a = 1.$
\item Properly color $\Sigma_2$ with $\pm 1, \ldots, \pm k_2$ using $\chi_2 - M_2$ colors. Let the $M_2$ unused colors be $y_1, \ldots, y_{M_2}.$ 
\item In $\Sigma_1$ perform a Type 2 replacement to replace $a$ with 0.
\item Make the following Type 1 replacements on $\Sigma_2.$
\begin{center}
\begin{tabular}{| c | c |}
\hline
Old Color & New Color \\ \hline
$-y_1$ & $x_1$ \\ \hline
$-y_2$ & $x_2$ \\ \hline
$\vdots$ & $\vdots$ \\ \hline
$-y_{M_2 -1}$ & $x_{M_2 -1}$ \\ \hline
$-y_{M_2}$ & $x_{M_2}$ \\
\hline
\end{tabular}
\end{center}
\item Use Type 4 replacements to recolor $\frac{M_1 - M_2 + 1}{2}$ more pairs of colors in $\Sigma_2$ using $x_{M_2+1}, \ldots, x_{M_1}$ and $a.$
\item Replace the remaining pairs of colors in $\Sigma_2$ using Type 3 replacements and colors $\pm (k_1 +1), \ldots, \pm(k_1 + k_2 - M_2 - \frac{M_1 - M_2 +1}{2}).$
\end{enumerate}

{\bf Case 2.1b:} Suppose $\Sigma_2$ is not exceptional. Then there exists a coloration using $\chi_2 - M_2$ colors and such that no edge has both endpoints colored $a,$ for some $a$ in the set of used colors. Call this coloration $\kappa_2.$ By Lemma \ref{even chromatic} we know that $-a$ cannot be in the deficiency set of $\kappa.$ Then color $\Sigma_1 \vee_+ \Sigma_2$ in the following way:

\begin{enumerate}
\item Properly color $\Sigma_1$ with $\pm 1, \ldots, \pm k_1$ using $\chi_1 - M_1$ colors. Let the $M_1$ unused colors be $x_1, \ldots, x_{M_1}.$ 
\item Color $\Sigma_2$ with $\kappa_2$ and colors $\pm 1, \ldots, \pm k_2$ Let the $M_2$ unused colors be $y_1, \ldots, y_{M_2}.$
\item In $\Sigma_2$ perform a Type 2 replacement to replace $a$ with 0.
\item Make the following Type 1 replacements on $\Sigma_2.$
\begin{center}
\begin{tabular}{| c | c |}
\hline
Old Color & New Color \\ \hline
$-y_1$ & $x_1$ \\ \hline
$-y_2$ & $x_2$ \\ \hline
$\vdots$ & $\vdots$ \\ \hline
$-y_{M_2 -1}$ & $x_{M_2-1}$ \\ \hline
$-y_{M_2}$ & $x_{M_2}$ \\ \hline
$-a$ & $x_{M_1}$ \\
\hline
\end{tabular}
\end{center}
\item Replace $\frac{M_1 - M_2 - 1}{2}$ more pairs of colors in $\Sigma_2$ using Type 4 replacements and colors $x_{M_2+1}, \ldots, x_{M_1 - 1}.$
\item Use Type 3 replacements to replace the remaining pairs of colors in $\Sigma_2$ using $\pm (k_1 +1), \ldots, \pm(k_1 + k_2 - M_2 - \frac{M_1 - M_2 -1}{2} - 1).$
\end{enumerate}

{\bf Case 2.2:} Suppose $\chi_1 = 2k_1 + 1$ and $\chi_2 = 2k_2$ for some positive integers $k_1$ and $k_2$. Then color $\Sigma_1 \vee_+ \Sigma_2$ in the following way:

\begin{enumerate}
\item Properly color $\Sigma_1$ with $\pm1, \ldots, \pm k_1,$ and 0 using $\chi_1-M_1$ colors. Let the $M_1$ unused colors be $x_1, \ldots, x_{M_1}.$ Note that 0 must be used.
\item Properly color $\Sigma_2$ with $\pm 1, \ldots, \pm k_2$ using $\chi_2 - M_2$ colors. Let the $M_2$ unused colors be $y_1, \ldots, y_{M_2}.$
\item Make the following Type 1 replacements on $\Sigma_2.$
\begin{center}
\begin{tabular}{| c | c |}
\hline
Old Color & New Color \\ \hline
$-y_1$ & $x_1$ \\ \hline
$-y_2$ & $x_2$ \\ \hline
$\vdots$ & $\vdots$ \\ \hline
$-y_{M_2}$ & $x_{M_2}$ \\
\hline
\end{tabular}
\end{center}
\item Use Type 4 replacements to recolor $\frac{M_1 - M_2 - 1}{2}$ more pairs of colors in $\Sigma_2$ using $x_{M_2 + 1}, \ldots, x_{M_1 - 1}$.
\item Of the remaining pairs of colors in $\Sigma_2$, let $c,-c$ be one. Use a Type 4 replacement to replace $c$ with $x_{M_1}$ and $-c$ with $k_1 + 1.$
\item In $\Sigma_1$ use a Type 1 replacement to replace 0 with $-(k_1 +1).$
\item Replace the remaining pairs of colors in $\Sigma_2$ using Type 3 replacements and colors $\pm (k_1 + 2), \ldots, \pm(k_1 + k_2 - M_2 - \frac{M_1 - M_2-1}{2}).$
\end{enumerate}

{\bf Case 2.3:} Suppose $\chi_2 = 2k_2 + 1$ for some positive integer $k_2.$ Then color $\Sigma_1 \vee_+ \Sigma_2$ in the following way:

\begin{enumerate}
\item Properly color $\Sigma_1$ with $\pm1, \ldots, \pm k_1,$ and possibly 0 using $\chi_1-M_1$ colors. Let the $M_1$ unused colors be $x_1, \ldots, x_{M_1}.$ If $\chi_1$ is odd, 0 must be used.
\item Properly color $\Sigma_2$ with $\pm 1, \ldots, \pm k_2,$ and 0 using $\chi_2 - M_2$ colors. Let the $M_2$ unused colors be $y_1, \ldots, y_{M_2}.$ Note that $0$ must be used.
\item Make the following Type 1 replacements on $\Sigma_2.$
\begin{center}
\begin{tabular}{| c | c |}
\hline
Old Color & New Color \\ \hline
$0$ & $x_{M_1}$ \\ \hline
$-y_1$ & $x_1$ \\ \hline
$-y_2$ & $ x_2$ \\ \hline
$\vdots$ & $\vdots$ \\ \hline
$-y_{M_2}$ & $x_{M_2}$ \\
\hline
\end{tabular}
\end{center}
\item Replace $\frac{M_1 - M_2 - 1}{2}$ more pairs of colors in $\Sigma_2$ using Type 4 replacements and colors $x_{M_2 + 1}, \ldots, x_{M_1 -1}$.
\item Use Type 3 replacements to recolor the remaining pairs of colors in $\Sigma_2$ using $\pm (k_1 + 1), \ldots, \pm(k_1 + k_2 - M_2 - \frac{M_1 - M_2 - 1}{2}).$
\end{enumerate}

This concludes the non-exceptional cases where $M_2 > 0$. 
\end{proof}

Lemmas \ref{exception 1}, \ref{exception 2}, \ref{main lemma 0}, and \ref{main lemma non-0} together prove Theorem \ref{main thm}. \hfill\qedsymbol

\end{document}